\documentclass[12pt]{amsart}

\usepackage{amssymb}

\usepackage{anysize}
\marginsize{2cm}{2cm}{2cm}{2cm}

\usepackage{versions}
\usepackage{hyperref}
\hypersetup{
  colorlinks   = true, 
  urlcolor     = blue, 
  linkcolor    = blue, 
  citecolor   = red 
}
\usepackage{cancel}
\usepackage{yhmath}
\usepackage{enumitem}
\usepackage{stackengine}

\newtheorem{theorem}{Theorem}[section]

\newtheorem{lemma}[theorem]{Lemma}

\newtheorem{corollary}[theorem]{Corollary}

\theoremstyle{definition}
\newtheorem{definition}[theorem]{Definition}

\theoremstyle{remark}
\newtheorem{remark}[theorem]{Remark}

\numberwithin{equation}{section}

\DeclareMathOperator{\vol}{vol}
\DeclareMathOperator{\sys}{sys}

\newcommand*{\where}{\ \ifnum\currentgrouptype=16 \middle\fi|\ }

\renewcommand{\epsilon}{\varepsilon}
\renewcommand{\phi}{\varphi}
\renewcommand{\kappa}{\varkappa}
\renewcommand{\theta}{\vartheta}

\title{Systolic inequalities and chromatic number}

\author{Alexander Kamal}
\author{Roman Karasev}

\address{Alexander Kamal, Moscow Institute of Physics and Technology, Institutskiy per. 9, Dolgoprudny, Russia 141700}
\email{kamal.am@phystech.edu}

\address{Roman Karasev, Institute for Information Transmission Problems RAS, Bolshoy Karetny per. 19, Moscow, Russia 127994}
\email{r\_n\_karasev@mail.ru}
\urladdr{http://www.rkarasev.ru/en/}

\keywords{Systolic inequality, Chromatic number, Odd cycles}
\subjclass[2010]{05C15, 05C12, 05C35, 51F99}

\begin{document}

\begin{abstract}
We show that the discrete versions of the systolic inequality that estimate the number of vertices of a simplicial complex from below have substantial applications to graphs, the one-dimensional simplicial complexes. Almost directly they provide good estimates for the number of vertices of a graph in terms of its chromatic number and the length of the smallest odd cycle. Combined with the graph-theoretic techniques of Berlov and Bogdanov, the systolic approach produces even better estimates.
\end{abstract}

\maketitle

\section{Introduction}

The systolic inequality in Riemannian geometry was introduced by Gromov~\cite{gromov1983}. In its modern form \cite{nabutovsky2019} it looks like
\[
\vol M \ge \frac{\sys M^n}{2^n n!}.
\]
Here $M$ is a closed $n$-dimensional Riemannian manifold (or even a piece-wise Riemannian polyhedron) with certain ``essentiality'' assumption, of which we only say that this is a property of the universal covering of $M$. 

The \emph{systole} of $M$, $\sys M$, is the shortest length of a non-contractible loop in $M$. Under a stronger than ``essentiality'' assumption of homological cup-length of $M$ equal to $n$, the factor $2^n$ in the denominator of the systolic inequality can be removed, see \cite[Theorem~1.18]{avvakumov2021systolic}.

In \cite{avvakumov2021systolic} an argument similar to the best known proofs of the systolic inequality was used to establish lower bounds on the number of vertices of simplicial complexes in terms of their ``combinatorial essentiality'' and the length of the shortest non-contractible loop along the edges (replacing the systole). 

In this note we show that a modification of those results works for the simplest non-trivial objects, the graphs ($1$-dimensional complexes), and improves some known estimates for their number of vertices in terms of their chromatic number and the minimal odd cycle length. Note that the minimal odd cycle length is generally larger than the minimal cycle (loop) length. Hence replacing the latter by the former in an estimate strengthens the result.

The direct consequence of the systolic technique is the following theorem.

\begin{theorem}[sys]
\label{theorem:systolic-chromatic}
Let $X$ be a graph with all odd cycles of length at least $2k+1$ and chromatic number $\chi$. Then the number of vertices of $X$ is at least
\[
2\binom{k-1 + \lfloor \frac{\chi-1}{2}\rfloor}{k-1} + \binom{k - 1 + \lfloor \frac{\chi-1}{2}\rfloor}{k} - 1.
\]
\end{theorem}

Mixing the systolic approach with observations in \cite{berlov2012} allows to further improve the estimate, e.g. in the following form.

\begin{theorem}[MIX-3]
\label{theorem:mix3}
Let $X$ be a graph with all odd cycles of length at least $2k+1$ and chromatic number $\chi>1$. Then the number of vertices of $X$ is at least
\[
\begin{cases}
   4\binom{k-1 + \lfloor \frac{\chi-1}{2}\rfloor}{\lfloor \frac{\chi-1}{2}\rfloor} + 2\binom{k-1 + \lfloor \frac{\chi-1}{2}\rfloor}{\lfloor \frac{\chi-1}{2}\rfloor - 1} - 1 &\text{if $\chi$ is odd,}\\
   4\binom{k-1 + \lfloor \frac{\chi-1}{2}\rfloor}{\lfloor \frac{\chi-1}{2}\rfloor} + 2\binom{k-1 + \lfloor \frac{\chi-1}{2}\rfloor}{\lfloor \frac{\chi-1}{2}\rfloor - 1} - 1 - \left( 2\binom{k-2 + \lfloor \frac{\chi-1}{2}\rfloor}{\lfloor \frac{\chi-1}{2}\rfloor} + \binom{k-2 + \lfloor \frac{\chi-1}{2}\rfloor}{\lfloor \frac{\chi-1}{2}\rfloor - 1} \right)  &\text{if $\chi$ is even.}
 \end{cases}
\]    
\end{theorem}

\subsection*{Acknowledgments} The authors thank Vladimir Dol'nikov for useful discussions and observations (given in Remarks \ref{remark:essentiality-chromatic} and \ref{remark:essentiality-forests}).

\section{Comparison to the known results}
\label{section:comparison}

We want to compare our new results to the results in \cite{berlov2012}. In order to do this we need to compare out notation and the notation of \cite{berlov2012}. Let us give the definition of the number $f(m,k)$ taken from \cite{berlov2012}.

\begin{definition}
Let $m, k$ be two positive integers. The number $f(m,k)$ is the maximum integer such that any graph $X$ with no odd cycles of length not exceeding $2k - 1$ (hence all odd cycles of length at least $2k+1$), and the number of vertices $|V(X)| \leq f(m,k)$, has a proper coloring in $m$ colors.
\end{definition}

A result of \cite{berlov2012} states that the following explicit estimate from below on the number $f(m,k)$ holds true
\[
f(m,k) \geq \frac{(m + k) \ldots (m+2k-1)}{2^{k - 1} k^k}.
\]

Let us compare it with Theorem \ref{theorem:systolic-chromatic}. Let $X$ be a graph such that $X$ has no odd cycles of length not exceeding $2k - 1$ (hence all odd cycles of length at least $2k+1$) and the chromatic number $\chi(X)>m$. Hence it has at least $f(m, k) + 1$ vertices by the definition of $f(m, k)$. In most cases we use this function $f$ as follows: \emph{Let $X$ be a graph such that $X$ has all odd cycles of length at least $2k+1$ and the chromatic number $\chi(X)$. Then it has at least $f(\chi(X)-1, k) + 1$ vertices.}

Therefore the result of \cite{berlov2012} provides an estimate on the number of vertices of such a graph:
\begin{equation}
\label{equation:bb-estimate}
|V(X)| \geq f(\chi - 1, k) + 1 \geq \frac{(\chi + k - 1) \ldots (\chi + 2k - 2)}{2^{k - 1} k^k} + 1.
\end{equation}

In order to make a comparison, let us slightly expand what Theorem \ref{theorem:systolic-chromatic} gives for graphs with all odd cycles of length at least $2k+1$:
\begin{multline*}
|V(X)| \geq 2\binom{k-1+ \lfloor \frac{\chi-1}{2}\rfloor}{k-1} + \binom{k-1 + \lfloor \frac{\chi-1}{2}\rfloor}{k} - 1 = \\ 
= 2 \frac{(\lfloor \frac{\chi-1}{2}\rfloor + 1) (\lfloor \frac{\chi-1}{2}\rfloor + 2) \ldots (k-1 + \lfloor \frac{\chi-1}{2}\rfloor)}{(k-1)!} 
+ \frac{\lfloor \frac{\chi-1}{2}\rfloor (\lfloor \frac{\chi-1}{2}\rfloor + 1)  \ldots (k-1 + \lfloor \frac{\chi-1}{2}\rfloor)}{k!} - 1 = \\ 
= \left(2k+ \lfloor \frac{\chi-1}{2}\rfloor \right) \frac{(\lfloor \frac{\chi-1}{2}\rfloor + 1) \ldots (k-1 + \lfloor \frac{\chi-1}{2}\rfloor)}{k!}.
\end{multline*}

This estimate is also polynomial of the same degree $k$ with respect to the chromatic number $\chi$, but the factor in front of the main term in the above formula is approximately $\frac{1}{2^{k} k!}$, which exceeds $\frac{1}{2^{k-1}k^k}$ in \eqref{equation:bb-estimate} (or \textbf{BB-2}) for any integer $k$ greater than $2$ and is exponentially better when $k\to\infty$.

The paper \cite{berlov2012} contains estimates other than what we have discussed above. Let us make some notations of the estimates found here, there, and obtained by mixing the both approaches.

\begin{eqnarray*}
    \textbf{sys} &=&  2\binom{k-1 + \lfloor \frac{\chi-1}{2}\rfloor}{k-1} + \binom{k-1 + \lfloor \frac{\chi-1}{2}\rfloor}{k} - 1;\\
    \textbf{BB-1} &=& \frac{(\chi + k - 1) \ldots (\chi + 2k - 2)}{2^{k - 1} k^k} + 1;\\
    \textbf{BB-2} &=& \frac{(\chi + k - 2) \ldots (\chi + 2k - 3)}{2^{k - 1} k^k} + (\chi - 1)(k-1) + 2;\\
    \textbf{BB-3} &=& (\chi - 1) + \frac{(k - 1)(\chi - 2)(\chi + 1)}{2} + 1;\\
    \textbf{MIX-1} &=& \frac{(\chi + k - 2) \ldots (\chi + 2k - 3)}{2^{k - 1} k^k} + 2\binom{k-2 + \lfloor \frac{\chi-1}{2}\rfloor}{\lfloor \frac{\chi-1}{2}\rfloor} + \binom{k-2 + \lfloor \frac{\chi-1}{2}\rfloor}{\lfloor \frac{\chi-1}{2}\rfloor - 1} + 1;\\
    \textbf{MIX-2} &=& (\chi - 2) + \frac{(k - 1)(\chi - 3)\chi}{2} + 2\binom{k-2 + \lfloor \frac{\chi-1}{2}\rfloor}{\lfloor \frac{\chi-1}{2}\rfloor} + \binom{k-2 + \lfloor \frac{\chi-1}{2}\rfloor}{\lfloor \frac{\chi-1}{2}\rfloor - 1} + 1;\\
    \textbf{MIX-3} &=& \begin{cases}
   4\binom{k-1 + \lfloor \frac{\chi-1}{2}\rfloor}{\lfloor \frac{\chi-1}{2}\rfloor} + 2\binom{k-1 + \lfloor \frac{\chi-1}{2}\rfloor}{\lfloor \frac{\chi-1}{2}\rfloor - 1} - 1 &\text{if $\chi$ is odd,}\\
   4\binom{k-1 + \lfloor \frac{\chi-1}{2}\rfloor}{\lfloor \frac{\chi-1}{2}\rfloor} + 2\binom{k-1 + \lfloor \frac{\chi-1}{2}\rfloor}{\lfloor \frac{\chi-1}{2}\rfloor - 1} - 1 - \left( 2\binom{k-2 + \lfloor \frac{\chi-1}{2}\rfloor}{\lfloor \frac{\chi-1}{2}\rfloor} + \binom{k-2 + \lfloor \frac{\chi-1}{2}\rfloor}{\lfloor \frac{\chi-1}{2}\rfloor - 1} \right)  &\text{if $\chi$ is even.}
     \end{cases}
\end{eqnarray*}

The estimate \textbf{sys} is the one from Theorem~\ref{theorem:systolic-chromatic}. The estimates \textbf{BB-x} are taken directly from~\cite{berlov2012}. The estimates \textbf{MIX-x} mix the systolic estimates for the size of a metric ball with lemmas from~\cite{berlov2012}. Two of them are stated below and Theorem~\ref{theorem:mix3} is stated in the introduction. The proofs are deferred to the next section.

\begin{corollary}[MIX-1]
\label{corollary:mix1}
Let $X$ be a graph with all odd cycles of length at least $2k+1$ and chromatic number $\chi$. Then the number of vertices of $X$ is at least
\[
\frac{(\chi + k - 2) \ldots (\chi + 2k - 3)}{2^{k - 1} k^k} + 2\binom{k-2 + \lfloor \frac{\chi-1}{2}\rfloor}{\lfloor \frac{\chi-1}{2}\rfloor} + \binom{k-2 + \lfloor \frac{\chi-1}{2}\rfloor}{\lfloor \frac{\chi-1}{2}\rfloor - 1} + 1.
\]    
\end{corollary}

\begin{corollary}[MIX-2]
\label{corollary:mix2}
Let $X$ be a graph with all odd cycles of length at least $2k+1$ and chromatic number $\chi$. Then the number of vertices of $X$ is at least
\[
(\chi - 2) + \frac{(k - 1)(\chi - 3)\chi}{2} + 2\binom{k-2 + \lfloor \frac{\chi-1}{2}\rfloor}{\lfloor \frac{\chi-1}{2}\rfloor} + \binom{k-2 + \lfloor \frac{\chi-1}{2}\rfloor}{\lfloor \frac{\chi-1}{2}\rfloor - 1} + 1
\]
\end{corollary}

\begin{table}[h!]
\begin{tabular}{l|l|l|l|l|l|l|l|l|l|}
\cline{2-10}
                                  & $k = 2$ & $k = 3$ & $k = 4$ & $k = 5$ & $k=6$ & $k = 7$ & $k = 8$ & $k = 9$ & $k = 10$ \\ \hline
\multicolumn{1}{|l|}{$\chi = 3$}  & BB-2    & BB-2    & BB-3    & BB-3    & BB-3  & BB-3    & BB-3    & BB-3    & BB-3     \\ \hline
\multicolumn{1}{|l|}{$\chi = 4$}  & BB-3    & BB-3    & BB-3    & BB-3    & BB-3  & BB-3    & BB-3    & BB-3    & BB-3     \\ \hline
\multicolumn{1}{|l|}{$\chi = 5$}  & BB-3    & BB-3    & BB-3    & BB-3    & BB-3  & sys     & sys     & sys     & sys      \\ \hline
\multicolumn{1}{|l|}{$\chi = 6$}  & BB-3    & BB-3    & BB-3    & BB-3    & BB-3  & BB-3    & BB-3    & BB-3    & BB-3     \\ \hline
\multicolumn{1}{|l|}{$\chi = 7$}  & BB-3    & BB-3    & BB-3    & sys     & sys   & sys     & sys     & sys     & sys      \\ \hline
\multicolumn{1}{|l|}{$\chi = 8$}  & BB-3    & BB-3    & BB-3    & BB-3    & BB-3  & sys     & sys     & sys     & sys      \\ \hline
\multicolumn{1}{|l|}{$\chi = 9$}  & BB-3    & BB-3    & BB-3    & sys     & sys   & sys     & sys     & sys     & sys      \\ \hline
\multicolumn{1}{|l|}{$\chi = 10$} & BB-3    & BB-3    & BB-3    & sys     & sys   & sys     & sys     & sys     & sys      \\ \hline
\multicolumn{1}{|l|}{$\chi = 11$} & BB-3    & BB-3    & sys     & sys     & sys   & sys     & sys     & sys     & sys      \\ \hline
\multicolumn{1}{|l|}{$\chi = 12$} & BB-3    & BB-3    & BB-3    & sys     & sys   & sys     & sys     & sys     & sys      \\ \hline
\multicolumn{1}{|l|}{$\chi = 13$} & BB-3    & BB-3    & sys     & sys     & sys   & sys     & sys     & sys     & sys      \\ \hline
\multicolumn{1}{|l|}{$\chi = 14$} & BB-3    & BB-3    & sys     & sys     & sys   & sys     & sys     & sys     & sys      \\ \hline
\multicolumn{1}{|l|}{$\chi = 15$} & BB-3    & BB-3    & sys     & sys     & sys   & sys     & sys     & sys     & sys      \\ \hline
\end{tabular}
\medskip
\caption{Comparison between \cite{berlov2012} and the systolic approach}
\end{table}

\begin{table}[]
\begin{tabular}{l|l|l|l|l|l|l|l|l|l|}
\cline{2-10}
                                  & $k = 2$ & $k = 3$ & $k = 4$ & $k = 5$ & $k=6$ & $k = 7$ & $k = 8$ & $k = 9$ & $k = 10$ \\ \hline
\multicolumn{1}{|l|}{$\chi = 3$}  & MIX-3   & MIX-3   & MIX-3   & MIX-3   & MIX-3 & MIX-3   & MIX-3   & MIX-3   & MIX-3    \\ \hline
\multicolumn{1}{|l|}{$\chi = 4$}  & BB-3    & BB-3    & BB-3    & BB-3    & BB-3  & BB-3    & BB-3    & BB-3    & BB-3     \\ \hline
\multicolumn{1}{|l|}{$\chi = 5$}  & MIX-3   & MIX-3   & MIX-3   & MIX-3   & MIX-3 & MIX-3   & MIX-3   & MIX-3   & MIX-3    \\ \hline
\multicolumn{1}{|l|}{$\chi = 6$}  & BB-3    & BB-3    & MIX-2   & MIX-2   & MIX-2 & MIX-2   & MIX-2   & MIX-2   & MIX-2    \\ \hline
\multicolumn{1}{|l|}{$\chi = 7$}  & MIX-3   & MIX-3   & MIX-3   & MIX-3   & MIX-3 & MIX-3   & MIX-3   & MIX-3   & MIX-3    \\ \hline
\multicolumn{1}{|l|}{$\chi = 8$}  & BB-3    & BB-3    & MIX-2   & MIX-2   & MIX-2 & MIX-3   & MIX-3   & MIX-3   & MIX-3    \\ \hline
\multicolumn{1}{|l|}{$\chi = 9$}  & BB-3    & MIX-3   & MIX-3   & MIX-3   & MIX-3 & MIX-3   & MIX-3   & MIX-3   & MIX-3    \\ \hline
\multicolumn{1}{|l|}{$\chi = 10$} & BB-3    & MIX-2   & MIX-2   & MIX-3   & MIX-3 & MIX-3   & MIX-3   & MIX-3   & MIX-3    \\ \hline
\multicolumn{1}{|l|}{$\chi = 11$} & BB-3    & MIX-3   & MIX-3   & MIX-3   & MIX-3 & MIX-3   & MIX-3   & MIX-3   & MIX-3    \\ \hline
\multicolumn{1}{|l|}{$\chi = 12$} & BB-3    & MIX-2   & MIX-3   & MIX-3   & MIX-3 & MIX-3   & MIX-3   & MIX-3   & MIX-3    \\ \hline
\multicolumn{1}{|l|}{$\chi = 13$} & BB-3    & MIX-3   & MIX-3   & MIX-3   & MIX-3 & MIX-3   & MIX-3   & MIX-3   & MIX-3    \\ \hline
\multicolumn{1}{|l|}{$\chi = 14$} & BB-3    & MIX-2   & MIX-3   & MIX-3   & MIX-3 & MIX-3   & MIX-3   & MIX-3   & MIX-3    \\ \hline
\multicolumn{1}{|l|}{$\chi = 15$} & BB-3    & MIX-3   & MIX-3   & MIX-3   & MIX-3 & MIX-3   & MIX-3   & MIX-3   & MIX-3    \\ \hline
\end{tabular}
\medskip
\caption{Comparison between all approaches, including the mixed one}
\end{table}

The estimate \textbf{BB-3} is better than the systolic estimate \cite{berlov2012} for small values of $\chi$ and $k$. For large values of $\chi$ and $k$, the estimates \textbf{BB-1} and \textbf{BB-2} are better than the degree $3$ polynomial \textbf{BB-3}, but the systolic estimate overcomes them.

When the mixed estimates are included, \textbf{MIX-2} sometimes improves the intermediate range of $\chi$ or $k$. The mixed approach \textbf{MIX-3} beats the systolic approach for large values of $\chi$ and $k$.

\section{Proofs of the mixed estimates}

In the proof of Theorem \ref{theorem:systolic-chromatic} and in \cite{berlov2012} one uses the notion of a metric ball in a graph.

\begin{definition}
Let the metric ball $B(x,i) \subseteq V(X)$ be the subset of vertices at edge-path distance at most $i$ from $x$. Let $\langle B(x,i) \rangle$ denote the subgraph of $X$ induced in the vertex set $B(x,i)$.
\end{definition}

For the integer $d(X,k-1) = \max_{x \in V(X)} |B(x,k-1)|$  (under the assumption that all odd cycles of $X$ have length at least $2k+1$) in \cite[Lemma~1]{berlov2012} it is proved that 
\begin{equation}
d(X, k-1) \ge (k-1) (\chi(X)-1) + 1. 
\end{equation}


Let us also make a definition of the estimate of the size of metric balls in terms of their chromatic number and odd cycles. It was implicit in \cite{berlov2012}, but here we make it explicit.

\begin{definition}
Let $d(m,k-1) = \min_{X}\max_{x \in V(X)} |B(x,k-1)|$, where the minimum is taken over graphs with all odd cycles of length at least $2k+1$ having no proper coloring in $m$ colors.
\end{definition}


The systolic approach (the estimate in Theorem~\ref{theorem:ball-graph} (a)) improves the estimate for the number of vertices of a metric ball of radius $k-1$ (under the assumption that all odd cycles of $X$ have length at least $2k+1$) to
\begin{equation}
\label{equation:ball-estimate-d}
|d(X, k-1)| \geq 2\binom{k-2 + \lfloor \frac{\chi(X)-1}{2}\rfloor}{\lfloor \frac{\chi(X)-1}{2}\rfloor} + \binom{k-2 + \lfloor \frac{\chi(X)-1}{2}\rfloor}{\lfloor \frac{\chi(X)-1}{2}\rfloor - 1},
\end{equation}
or more generally
\begin{equation}
\label{equation:ball-estimate-dmk}
|d(m, k-1)| \geq 2\binom{k-2 + \lfloor \frac{m}{2}\rfloor}{\lfloor \frac{m}{2}\rfloor} + \binom{k-2 + \lfloor \frac{m}{2}\rfloor}{\lfloor \frac{m}{2}\rfloor - 1}.
\end{equation}

In order to combine the systolic approach and \cite[Lemma~2]{berlov2012} we need to restate \cite[Lemma~2]{berlov2012} in our terms.

\begin{lemma}[Rephrased Lemma~2 of \cite{berlov2012}]
\label{lemma:bb2}
\begin{equation}
\label{equation::bb-estimate3}
f(m,k) \geq f(m - 1, k) + d(m,k-1).
\end{equation}
\end{lemma}
\begin{proof}
This is essentially the proof of \cite[Lemma~2]{berlov2012}, given here for completeness. Assume the contrary: $f(m,k) \leq f(m - 1, k) + d(m,k-1) - 1$. This means that there exists a graph $X$ with $f(m - 1, k) + d(m,k-1)$ vertices with all odd cycles of length at least $2k+1$ and no proper coloring in $m$ colors. 

By definition of $d(m,k-1)$ there exists a metric ball $B(x,k-1)$ of $X$ with at least $d(m,k-1)$ vertices. Then the complement of this ball (as a set of vertices) induces a graph $X'$ with at most $f(m-1,k)$ vertices. Since $X'$ obviously has all odd cycles of length at least $2k+1$, the definition of $f(m-1,k)$ implies that it is properly colorable in $m-1$ colors. 

Color the vertices of $B(x,k-1)$ so that the vertices at distance $k-1,k-3,k-5,\ldots$ from $x$ are colored in the new color $m$, others are colored in one of the old colors. The ball is then properly colored since its vertices of the same color could only be connected by forming an odd cycle of length at most $2k-1$. The subgraph $X'$ is already properly colored, and the edges between the ball and $X'$ are all properly colored since they are colored in one of the old colors and the new color $m$.

This means that $X$ has a proper coloring in $m$ colors, a contradiction. 
\end{proof}

We extract a part of the above proof that we will need separately as the following.
\begin{lemma}
\label{lemma-even-balls}
Let $X$ be a graph with all odd cycles of length at least $2k+1$. Then any metric ball in $X$ of radius $k-1$ contains no odd cycle.
\end{lemma}
\begin{proof}
The coloring of the ball in two colors from the previous proof is proper and shows that there is no odd cycle inside this ball.
\end{proof}

\begin{proof}[Proof of Corollary~\ref{corollary:mix1}]
Using \cite[Lemma~2]{berlov2012} and \cite[Theorem~3]{berlov2012} we can get an estimate:
\[
\frac{(\chi + k - 2) \ldots (\chi + 2k - 3)}{2^{k - 1} k^k} + d(X,k) + 1.
\]
Then we estimate $d$ using \eqref{equation:ball-estimate-d}
\[
|V(X)| \geq \frac{(\chi + k - 2) \ldots (\chi + 2k - 3)}{2^{k - 1} k^k} + 2\binom{k-2 + \lfloor \frac{\chi-1}{2}\rfloor}{\lfloor \frac{\chi-1}{2}\rfloor} + \binom{k-2 + \lfloor \frac{\chi-1}{2}\rfloor}{\lfloor \frac{\chi-1}{2}\rfloor - 1} + 1.
\]
\end{proof}

\begin{proof}[Proof of Corollary~\ref{corollary:mix2}]
We use \cite[Lemma~2]{berlov2012} and \cite[Theorem~2]{berlov2012} to get an estimate:
\[
|V(X)|\ge (\chi - 2) + \frac{(k - 1)(\chi - 3)\chi}{2} + d(X,k) + 1.
\]
Then we estimate $d(X,k)$ using \eqref{equation:ball-estimate-d}
\[
|V(X)| \geq (\chi - 2) + \frac{(k - 1)(\chi - 3)\chi}{2} + 2\binom{k-2 + \lfloor \frac{\chi-1}{2}\rfloor}{\lfloor \frac{\chi-1}{2}\rfloor} + \binom{k-2 + \lfloor \frac{\chi-1}{2}\rfloor}{\lfloor \frac{\chi-1}{2}\rfloor - 1} + 1
\]
\end{proof}

\begin{proof}[Proof of Theorem~\ref{theorem:mix3}]
Using Lemma~\ref{lemma:bb2} and the systolic estimate for the number of vertices of a metric ball \eqref{equation:ball-estimate-dmk}, we recursively estimate the function $f$:
\begin{multline*}
        f(m,k) \geq 1 + \sum \limits_{c = 2}^m d(c, k-1) \geq \sum \limits_{c = 1}^m \left(  2\binom{k-2 + \lfloor \frac{c}{2}\rfloor}{\lfloor \frac{c}{2}\rfloor} + \binom{k-2 + \lfloor \frac{c}{2}\rfloor}{\lfloor \frac{c}{2}\rfloor - 1} \right) = \\ = 
\begin{cases}
   4\binom{k-1 + \lfloor \frac{m}{2}\rfloor}{\lfloor \frac{m}{2}\rfloor} + 2\binom{k-1 + \lfloor \frac{m}{2}\rfloor}{\lfloor \frac{m}{2}\rfloor - 1} - 2 &\text{if $m$ is even,}\\
   4\binom{k-1 + \lfloor \frac{m}{2}\rfloor}{\lfloor \frac{m}{2}\rfloor} + 2\binom{k-1 + \lfloor \frac{m}{2}\rfloor}{\lfloor \frac{m}{2}\rfloor - 1} - 2 - \left( 2\binom{k-2 + \lfloor \frac{m}{2}\rfloor}{\lfloor \frac{m}{2}\rfloor} + \binom{k-2 + \lfloor \frac{m}{2}\rfloor}{\lfloor \frac{m}{2}\rfloor - 1} \right)  &\text{if $m$ is odd.}
\end{cases}
\end{multline*}
     
If we take $m=\chi - 1$ in the above formulas, we obtain the lower bound for the number of vertices under the hypothesis of the theorem:
\begin{multline*}
|V| \geq f(\chi - 1, k) + 1 \geq \\
\geq \begin{cases}
   4\binom{k-1 + \lfloor \frac{\chi-1}{2}\rfloor}{\lfloor \frac{\chi-1}{2}\rfloor} + 2\binom{k-1 + \lfloor \frac{\chi-1}{2}\rfloor}{\lfloor \frac{\chi-1}{2}\rfloor - 1} - 1 &\text{if $\chi$ is odd,}\\
   4\binom{k-1 + \lfloor \frac{\chi-1}{2}\rfloor}{\lfloor \frac{\chi-1}{2}\rfloor} + 2\binom{k-1 + \lfloor \frac{\chi-1}{2}\rfloor}{\lfloor \frac{\chi-1}{2}\rfloor - 1} - 1 - \left( 2\binom{k-2 + \lfloor \frac{\chi-1}{2}\rfloor}{\lfloor \frac{\chi-1}{2}\rfloor} + \binom{k-2 + \lfloor \frac{\chi-1}{2}\rfloor}{\lfloor \frac{\chi-1}{2}\rfloor - 1} \right)  &\text{if $\chi$ is even.}
 \end{cases}
\end{multline*}
\end{proof}

\begin{remark}
The proof of Theorem \ref{theorem:base-case-gen} (implying Theorem~\ref{theorem:systolic-chromatic}) also uses certain induction. But Theorem \ref{theorem:mix3} wins because in the case of a graph $X$ it has approximately $\chi(X)$ induction steps, while Theorem \ref{theorem:base-case-gen} has approximately $\chi(X)/2$ induction steps.
\end{remark}

\section{Discrete systolic inequalities}
\label{section:discrete-sys}

Our proof of Theorem \ref{theorem:systolic-chromatic} is a small modification of the proof of \cite[Theorem~1.4]{avvakumov2021systolic}. Let us recall the crucial notions from that paper.

\begin{definition}
The \emph{(edge-path) systole} $\sys X$ of a simplicial complex $X$ is the smallest integer such that any closed path along the edges of $X$ of edge-length less than $\sys X$ is null-homotopic. If every component of $X$ is simply connected, the systole is, by convention, $+\infty$.
\end{definition}

Note that for a graph $X$ (a $1$-dimensional simplicial complex) the systole $\sys X$ is just the length of the shortest cycle, since no cycle in a graph is contractible.

\begin{definition}
A subset $Y\subset V(X)$ is called \emph{inessential} if the natural map $\pi_1(C)\to \pi_1(D)$ is trivial for every connected component $C$ of $\langle Y\rangle$ and every connected component $D$ of $X$.
\end{definition}

Note that an inessential subset of the vertices of a graph is a subset on which the graph induces a forest.

\begin{definition}
\label{definition:combinatorially-n-essential}
A complex $X$ is called \emph{combinatorially $n$-essential} if its vertex set cannot be partitioned into $n$ inessential sets or fewer.
\end{definition}

\begin{theorem}[Avvakumov, Balitskiy, Hubard, Karasev, \cite{avvakumov2021systolic}]
\label{theorem:a-la-gromov}
Let $X$ be a combinatorially $n$-essential simplicial complex, $n\ge 1$. Then the number of vertices of $X$ is at least
\[
\binom{n + \left\lfloor\frac{\sys X}{2}\right\rfloor - 1}{n-1} + 2 \binom{n + \left\lfloor\frac{\sys X}{2}\right\rfloor - 1}{n} - 1 \ge \binom{n + \left\lfloor\frac{\sys X}{2} \right\rfloor}{n} \ge \frac{1}{n!} \left\lceil\frac{\sys X}{2} \right\rceil^n.
\]
\end{theorem}

This theorem may be directly applied to graphs, but unlike the estimates we are aiming at, it does not distinguish between the even and odd cycles. In the next section we explain the appropriate modification.

\section{Systolic inequalities relative to a covering}
\label{section:proofs}

In order to handle cycles of odd length and ignore cycles of even length, we need a generalization of the discrete systolic inequality, also from \cite{avvakumov2021systolic}. If one is only interested in the case of graphs then it makes sense to read the following general statement replacing ``simplicial complex'' by ``graph''.

\begin{definition}
Assume that a simplicial complex $X$ has a covering map $\pi : \widetilde X\to X$. A subset $Y\subset V(X)$ is \emph{$\pi$-inessential} if the restriction $\pi^{-1}(\langle Y\rangle) \to \langle Y\rangle$ is a trivial cover. We assume that a cover over $\langle Y \rangle$ is trivial if and only if it is trivial (equal to $C\times D$ with a discrete set $D$) over every connected component $C$ of $\langle Y \rangle$.
\end{definition}

\begin{definition}
\label{definition:n-ess-cover}
Assume that a simplicial complex $X$ has a covering map $\pi : \widetilde X\to X$. This covering map is called combinatorially $n$-essential if the vertex set of $X$ cannot be partitioned into $n$ or fewer $\pi$-inessential sets.
\end{definition}

\begin{definition}
\label{definition:triviality-radius}
We say that a covering map $\pi : \widetilde X\to X$ has \emph{homotopy triviality radius $r$} if every metric ball $B(x,r) \subseteq V(X)$ is $\pi$-inessential and $r$ is the maximum integer with this property.
\end{definition}

We now use the following theorem.

\begin{theorem}[Avvakumov, Balitskiy, Hubard, Karasev, \cite{avvakumov2021systolic}]
\label{theorem:base-case-gen}
Let a simplicial complex $X$ have a covering map $\pi : \widetilde X\to X$ and let $\pi$ be combinatorially $n$-essential. Let $r$ be the homotopy triviality radius of $\pi$. Then there exists a vertex $x\in X$ such that for any $i=0,\ldots,r+1$  the number of vertices in the ball $B(x,i)$ is at least $b_n(i)$, where positive integers $b_n(i)$ satisfy the following recursive relations:
\begin{itemize}
\item $b_1(i) = 2i+1$ for any $i=0,\ldots,r$ and $b_1(r+1)= 2r+2$;
\item $b_n(i) = \sum_{0\le j\le i}b_{n-1}(j)$ for any $i=0,\ldots,r+1$.
\end{itemize}
In particular, $b_n(i) \ge 2\binom{i + n -1}{n} + \binom{i + n-1}{n-1}$ for any $i=0,\ldots,r$, and $b_n(r+1)\ge 2\binom{r + n}{n} + \binom{r + n}{n-1} - 1$.
\end{theorem}

Its particular case for graphs, odd cycles, and chromatic number is the following.

\begin{theorem}
\label{theorem:ball-graph}
Let $X$ be a graph with all odd cycles of length at least $2k+1$ and chromatic number $\chi$. Then 

a$)$ some metric ball of radius $k-1$ in $X$ has at least 
\[
b_n(k-1) \ge 2\binom{k-2+\left\lfloor \frac{\chi-1}{2}\right\rfloor}{k-2} + \binom{k-2+\left\lfloor \frac{\chi-1}{2}\right\rfloor}{k-1} = 2\binom{k-2+\left\lfloor \frac{\chi-1}{2}\right\rfloor}{\left\lfloor \frac{\chi-1}{2}\right\rfloor} + \binom{k-2+\left\lfloor \frac{\chi-1}{2}\right\rfloor}{\left\lfloor \frac{\chi-1}{2}\right\rfloor-1}
\]
vertices; 

b$)$ and some metric ball of radius $k$ in $X$ has at least 
\[
b_n(k)\ge 2\binom{k-1+\left\lfloor \frac{\chi-1}{2}\right\rfloor}{k-1} + \binom{k-1+\left\lfloor \frac{\chi-1}{2}\right\rfloor}{k} - 1
\]
vertices.
\end{theorem}

\begin{proof}[Proof of Theorem \ref{theorem:systolic-chromatic} assuming Theorem~\ref{theorem:ball-graph}]
This is actually Theorem~\ref{theorem:ball-graph} (b).
\end{proof}

\begin{proof}[Proof of Theorem \ref{theorem:ball-graph} assuming Theorem~\ref{theorem:base-case-gen}]
First, it makes sense to pass to a connected component of the graph $X$ in order to speak about its fundamental group $G:=\pi_1(X)$. Consider its (normal) subgroup $H\subseteq G$ consisting of the based cycles (called \emph{loops} in topology) of even length. In other words, $H$ is the kernel of the map $G\mapsto \mathbb F_2$ sending a loop to the parity of its length. The universal cover $\widetilde X\to X$ then produces the cover $\overline X = \widetilde X/H$ with its double covering map $\pi : \overline X\to X$. This is the covering map to which we apply the above definitions and Theorem~\ref{theorem:base-case-gen}. Note that a change of the base point does not change this covering, because adding a path in front of a cycle and the same reversed path after the cycle keeps the parity of its length.

Note that a cycle $L\subset G$ can be lifted to $\overline X$ as a cycle if and only if it has even length. Therefore a subset $Y\subset V(X)$ is $\pi$-inessential if and only if it induces no odd cycle. Equivalently, $Y$ induces a bipartite subgraph of $X$.

Definition \ref{definition:n-ess-cover} then means that $V(X)$ cannot be partitioned into $n$ or fewer sets that induce bipartite subgraphs. From the hypothesis of the theorem we know that the chromatic number is $\chi$ and want to bound $n$ from below. Evidently, if Definition \ref{definition:n-ess-cover} fails for some $n$ then we may color each of $n$ or fewer induced bipartite subgraphs in two colors, totally spending $2n$ colors for a proper coloring of $X$. This implies $2n\ge \chi$.
The opposite inequality
\[
2n < \chi\Leftrightarrow 2n \le \chi - 1 \Leftrightarrow n \le \left\lfloor \frac{\chi-1}{2}\right\rfloor 
\] 
therefore implies the validity of Definition \ref{definition:n-ess-cover}. In particular, $X$ is $n$-essential with $n = \left\lfloor \frac{\chi-1}{2}\right\rfloor$.

The hypothesis that every odd cycle has length at least $2k+1$ and Lemma~\ref{lemma-even-balls} imply that we may use the radius $r=k-1$ in Definition \ref{definition:triviality-radius} for $\pi$. Now we apply Theorem \ref{theorem:base-case-gen} and obtain the two needed estimates.
\end{proof}

\begin{remark}[Hinted by Vladimir Dol'nikov]
\label{remark:essentiality-chromatic}
In the above proof we have established the relation between the essentiality of a graph $X$ and its chromatic number as
\[
n(X) \ge \left\lfloor \frac{\chi(X)-1}{2}\right\rfloor = \left\lceil \frac{\chi(X)-2}{2}\right\rceil = \left\lceil \frac{\chi(X)}{2}\right\rceil - 1.
\]
Assume now that $X$ is colored in $\chi(X)$ colors. Joining these colors in pairs (one may have no pair if the chromatic number is odd) we partition its vertices into $n = \lceil \chi(X)/2\rceil$ parts so the every part induces a bipartite graph. That is, each of the $n$ parts induces either a bipartite graph or a graph with no edges, each part therefore having no odd cycles. This implies that
\[
n(X) < \left\lceil \chi(X)/2\right\rceil \Leftrightarrow n(X) \le  \left\lceil \frac{\chi(X)}{2}\right\rceil - 1.
\]
In view of these two inequalities the essentiality of a graph (with respect to odd cycles) is in fact a function of its chromatic number, $n(X) = \left\lceil \frac{\chi(X)}{2}\right\rceil - 1$.
\end{remark}

\begin{remark}[Hinted by Vladimir Dol'nikov]
\label{remark:essentiality-forests}
Theorem \ref{theorem:a-la-gromov} can be applied to a graph $X$ directly, not distinguishing the odd and even cycles. Then the number of vertices of a graph will be estimated from below based on the length of its smallest cycle $\sys X$ and $n'=n'(X)$, the smallest $n'$ such that the vertices of $X$ can be partitioned into $n'+1$ parts so that every part induces a forest. Since $\sys X$ is no greater than the length of the smallest odd cycle, this estimate may improve Theorem \ref{theorem:systolic-chromatic} only if $n'(X) > n(X) = \left\lceil \frac{\chi(X)}{2}\right\rceil - 1$. This is so for sufficiently big bipartite and tripartite graphs, but Theorem \ref{theorem:a-la-gromov} is not very interesting in this case since those graphs do contain short cycles. 
\end{remark}

\section{Appendix}

In order to make this exposition self-contained, we literally reproduce the proof of Theorem~\ref{theorem:base-case-gen} from \cite{avvakumov2021systolic}. The term \emph{loop} here is used for a cycle in a graph.

Let us first prove the theorem for $n=1$. In that case, consider a simple (passing once through each of its vertices) closed loop, whose lift to $\widetilde X$ is not closed. Such a loop exists since otherwise $X$ would be inessential. The loop has at least $2r+2$ vertices, because otherwise it would fit into a ball of radius $r$ containing only loops lifting to loops. The equalities $b_1(i) = 2i+1$ for any $i=0,\ldots,r$ and $b_1(r+1)= 2r+2$ then easily follow.

We proceed by induction over $n$. Assume that $n>1$ and that the theorem is already proven for smaller $n$. Find a smallest subset of vertices $Z\subset V(X)$ such that its complement $Y:=V(X)\setminus Z$ is $\pi$-inessential.

By definition, the restriction of the covering map, $\pi^{-1}(\langle Z\rangle)\to \langle Z\rangle$, is then combinatorially $(n-1)$-essential. The metric balls of $Z$ are contained in the respective metric balls of $X$ and the $\pi$-inessentiality property of the balls is therefore preserved. Hence the homotopy triviality radius of the restriction of the covering map is no smaller than that of $\pi$.

By the induction hypothesis, there exists a vertex $x\in Z$ such that for any $i=0,\ldots,r+1$ the number of vertices in $B(x,i)\cap Z$ is at least $b_{n-1}(i)$.

Consider the following modification of $Z$:
\[
Z':=Z\cup S(x,r+1)\setminus B(x, r).
\]
Let us prove that $Y':=V(X)\setminus Z'$ is $\pi$-inessential.

Assume to the contrary, that there is a simple closed edge-path $P$ in $\langle Y' \rangle$ whose lift to $\widetilde X$ is not closed. The set of vertices $S(x,r+1)$ separates the $1$-skeleton of $X$ into two disconnected components with vertex sets $B(x,r)$ and $V(X)\setminus B(x,r+1)$, respectively. Since $P$ is simple and does not contain vertices of $S(x,r+1)$, we have that
\begin{itemize}
\item
either $P\subset \langle B(x,r)\rangle$;
\item
or $P\subset \langle Y'\setminus B(x,r+1)\rangle$.
\end{itemize}
In the first case, $P$ lifts to a closed path in $\widetilde X$ because $\pi$ has the homotopy triviality radius $r$. In the second case, $P$ lifts to a closed path in $\widetilde X$ because $Y'\setminus B(x,r+1)\subseteq Y$ and $Y$ is $\pi$-inessential.

Since $Y'$ is $\pi$-inessential, from the minimality of $Z$, we obtain that
\[
|Z| \leq |Z'|,
\]
and therefore
\[
|S(x,r+1)\setminus Z| \ge |B(x,r)\cap Z|.
\]
Analogously, for every $i \leq r$ we obtain that
\[
|S(x,i+1)\setminus Z| \ge |B(x,i)\cap Z|.
\]

Finally, for any $i=0,\ldots,r+1$ we have that
\begin{multline}
\label{equation:ball-estimate}
|B(x,i)| = |B(x,i)\cap Z| + \sum_{0\le j\le i}|S(x,j)\setminus Z| \ge \\
\ge \sum_{0\le j\le i}|B(x,j)\cap Z| \ge \sum_{0\le j\le i} b_{n-1}(j)=b_n(i).
\end{multline}

It remains to prove the inequality $b_n(i) \ge 2\binom{i + n -1}{n} + \binom{i + n-1}{n-1}$. The case $n=1$ reads
\[
b_1(i) \ge 2i + 1,
\]
and has been already shown in the beginning of the proof.

For $n>1$ we have:
\begin{multline*}
b_n(i) = \sum_{0\le j\le i}b_{n-1}(j) \ge\\
\ge \sum_{0\le j\le i} \left( 2\binom{j + n - 2}{n-1} + \binom{j + n - 2}{n-2}\right) = 2\binom{i + n -1}{n} + \binom{i + n-1}{n-1}.
\end{multline*}

For $b_n(r+1)$, as compared to $b_n(i)$ with $i=0,\ldots,r$, the estimate is $1$ less in case $n=1$, as was described in the beginning of the proof. In the course of summation this $-1$ summand carries on, hence we have
\[
b_n( r+1) \ge 2\binom{r + n}{n} + \binom{r + n}{n-1} - 1.
\]

\bibliography{../Bib/karasev}
\bibliographystyle{abbrv}

\end{document}